\documentclass{amsart}
\usepackage{latexsym, amssymb,amsthm,amsmath,verbatim,xcolor}

\begin{document}

\newtheorem{theorem}{Theorem}[section]
\newtheorem{proposition}[theorem]{Proposition}
\newtheorem{example}[theorem]{Example}
\newtheorem{question}[theorem]{Question}
\newtheorem{definition}[theorem]{Definition}
\newtheorem{corollary}[theorem]{Corollary}
\newtheorem{lemma}[theorem]{Lemma}
\newtheorem{claim}[theorem]{Claim}

\newcommand{\pr}[1]{\left\langle #1 \right\rangle}
\newcommand{\mH}{\mathcal{H}}
\newcommand{\mK}{\mathcal{K}}
\newcommand{\uhr}{\upharpoonright}
\newcommand{\cl}{\operatorname{cl}}
\newcommand{\mR}{\mathcal{R}}
\newcommand{\mG}{\mathcal{G}}
\newcommand{\mA}{\mathcal{A}}
\newcommand{\mV}{\mathcal{V}}
\newcommand{\mU}{\mathcal{U}}
\newcommand{\mP}{\mathcal{P}}
\newcommand{\mB}{\mathcal{B}}
\newcommand{\mI}{\mathcal{I}}
\newcommand{\mW}{\mathcal{W}}
\newcommand{\mM}{\mathcal{M}}
\newcommand{\C}{\mathrm{C}}
\newcommand{\mY}{\mathcal{Y}}
\newcommand{\mO}{\mathcal{O}}
\newcommand{\mC}{\mathcal{C}}
\newcommand{\D}{\mathrm{D}}
\renewcommand{\O}{\mathrm{O}}
\newcommand{\K}{\mathrm{K}}
\newcommand{\OD}{\mathrm{OD}}
\newcommand{\Do}{\D_\mathrm{o}}
\newcommand{\sone}{\mathsf{S}_1}
\newcommand{\gone}{\mathsf{G}_1}
\newcommand{\sfin}{\mathsf{S}_\mathrm{fin}}
\newcommand{\gfin}{\mathsf{G}_\mathrm{fin}}
\newcommand{\Em}{\longrightarrow}
\newcommand{\menos}{{\setminus}}
\newcommand{\w}{{\omega}}
\newcommand{\PP}{{\mathbb{P}}}
\newcommand{\QQ}{{\mathbb{Q}}}

\title{Internal characterizations of productively Lindel\"of spaces}
\author[L. F. Aurichi]{Leandro F. Aurichi$^1$}
\thanks{$^1$ The first author was partially supported by FAPESP (2013/05469-7 and 2015/25725).
A part of the results were obtained during the visit of the first author of the Kurt G\"odel Center
at the University of Vienna in January 2017 partially supported by the FWF Grant M 1851-N35.}
\address{Instituto de Ci\^encias Matem\'aticas e de Computa\c{c}\~ao,
Universidade de S\~ao Paulo, Caixa Postal 668,
S\~ao Carlos, SP, 13560-970, Brazil}
\email{aurichi@icmc.usp.br}

\author[L. Zdomskyy]{Lyubomyr Zdomskyy$^2$}
\thanks{$^2$ The second author would like to thank the Austrian Science Fund FWF
(Grants I 1209-N25 and
I 2374-N35) for generous support for this research.}
\address{Kurt Goedel Research Center for Mathematical Logic,
University of Vienna,
Waehringer Strasse 25, A-1090 Wien, Austria}
\email{lzdomsky@gmail.com}
\urladdr{http://www.logic.univie.ac.at/~lzdomsky/}
\date{}

\begin{abstract}
  We present an internal characterization for the productively Lindel\"of property, thus answering a long-standing problem attributed to Tamano. We also present some results about the relation ``Alster spaces'' \emph{vs.} ``productively Lindel\"of spaces''.
\end{abstract}

\maketitle

\section{Introduction}

We say that a topological space is Lindel\"of if every open covering
for it has a countable subcovering. We say that a Lindel\"of space
$X$ is  productively Lindel\"of  if $X \times Y$ is Lindel\"of for
every Lindel\"of space $Y$. This is a class that contains all
$\sigma$-compact spaces but we do not know yet much more about which
other spaces are in it. For example, it is not know if the space of
the irrationals is productively Lindel\"of  - although  consistently
it is not, see \cite{Moo99} and references therein. This is the
famous Michael's problem.

As with the case of the irrationals, most of the known results about productively Lindel\"of spaces use some kind of
combinatorial hypothesis beyond ZFC (see e.g. \cite{AlaAurJunTal11, Moo99}). So this  property seems  to have a set-theoretic nature.

The main result of the second section is Theorem~\ref{main theorem}
which gives  an internal characterization
 of the class of productively Lindel\"of spaces and thus  solves a problem attributed to H. Tamano in \cite{Przymusinski}.
 The formulation of the property is combinatorial in the sense that it looks like a diagonalization
property. Basically, it says that for a regular space $X$, $X$ is
productively Lindel\" of if, and only if, for every collection $\mV$
of open coverings of $X$ that is ``small enough'', there is a
countable collection $\mC$ of open sets such that $V \cap \mC$ is
still a covering for $X$ for every $V \in \mV$. Here it is important
to stress that we are talking about arbitrary collection of open
coverings: It is not enough to use only open coverings made by
elements of a fixed open base since then the conclusion in
Theorem~\ref{main theorem} for second countable spaces is simply
trivial. However, following the proof presented here, one could use
only sets of the form ``basic open set minus two points''. So far,
we were not able to use this characterization to solve the Michael's
problem.

In the third section we present some new results about the relation ``Alster spaces'' (defined below) \emph{vs} ``productively Lindel\"of spaces''.
In particular, we obtain the original Alster's result (\cite{Alster1986}) as a corollary. This gives a simplification compared to the original proof.
Alster's result is probably the best known one regarding productively Lindel\" of spaces in general. The result presented here has the advantage that
the set-theoretic assumption in it is much weaker than CH.

Finally, in the fourth section, we present a new property which is (formally) in between Alster and productively Lindel\"of ones. One of the applications
 of this property is that it makes  the relation between the properties of Alster and Hurewicz clearer.

\section{An internal characterization for productively Lindel\"of spaces}

The internal formulation (pL) is presented in the first subsection, where it is also showed that if
 the space is productively Lindel\"of, then pL holds. Then, in the second subsection, we show that pL, together with a certain technical property, is enough to prove that a space is productively
Lindel\"of. After that, in the third subsection, we finally show that every regular space has
that technical property, concluding the main result (Theorem \ref{main theorem}). Finally, in the
 last subsection we discuss some cases when this characterization can be extended for spaces that are not regular.

\subsection{A topology over open coverings}

Let $(X, \tau)$ be a topological space. We denote by $\O$ the
collection of all open coverings of $X$. Given a finite $F \subset
\tau$, we use the following notation:
$$F^* = \{\mU \in \O: F \subset \mU\}.$$
Note that, since $F^* \cap G^* = (F \cup G)^*$, the collection
$\{F^*: F \subset \tau$ is finite$\}$ is a base for a topology over
$\O$. Also, note that for any $\mU \in \O$, $\{F^*: F \subset \mU$
is finite$\}$ is a local base for $\mU$. Therefore, the following
definition is just a translation of the Lindel\"ofness property of
subsets of $\O$:

\begin{definition}
Let $(X, \tau)$ be a topological space. We say that   $\mY \subset
\O$ is a {\bf Lindel\"of collection} if, for every $f: \mY
\rightarrow [\bigcup \mY]^{<\w}$ such that $f(\mathcal U)\subset \mathcal
U$ for all $\mathcal U\in\mathcal Y$,  there is a sequence
$(\mU_n)_{n \in \w}$ of coverings in $\mY$ such that for every $\mU \in \mY$, there is an
$n \in \w$ such that $f(\mU_n) \subset \mU$.
\end{definition}

If we cover $X \times \mathcal Y$ in  such a way that each $(x, \mU)
\in A \times \{A\}^*$ for some $A \in \mU$ with $x \in A$, then the
Lindel\"of property of $X \times \mY$ would imply that there is a
sequence $(A_n)_{n \in \omega}$ that is enough to cover the whole of $X
\times \mY$. This motivates the following definition:

\begin{definition}
  Let $(X, \tau)$ be a topological  space. We say that $X$
   has the {\bf pL property} if, for every Lindel\"of collection
   $\mY \subset \O$, there is a sequence $(A_n)_{n \in \w}$ of open
   sets
   such that, for each $(x, \mU) \in X \times \mY$ there is some
   $n \in \w$ such that $x \in A_n \in \mU$.
\end{definition}

After the comments above, it is easy to see the following:

\begin{proposition}\label{easy part}
  Let $(X, \tau)$ be a topological space. If $X$ is
  productively Lindel\"of, then $X$ has the pL property.
\end{proposition}

In the following, we will discuss when the pL property implies a
space is productively Lindel\"of.

\subsection{From $Y$ to $\mY$}

Let $(X, \tau)$ and $(Y, \rho)$ be topological spaces.  Let $\mW$ be
an open covering of $X \times Y$ made by basic open sets. For each
$y \in Y$, define $\mU_y = \{U \in \tau:$ there is a $V \in \rho$
such that $U \times V \in \mW$ and $y \in V\}$. Note that each
$\mU_y$ is an element of $\O$. We define $\mY = \{\mU_y: y \in Y\}$
with the topology defined as before. Note that $\mY$ depends on
$\mW$, but we will not mark this dependency unless necessary.

\begin{proposition}\label{mY is Lindelof}
  If $\mW$ is an open covering of $X \times Y$ and $Y$ is Lindel\"of,
  then $\mY$ is Lindel\"of.
\end{proposition}
\begin{proof}
  For each $\mU_y$,  let $\{U_1^y, ..., U_{n_y}^y\} \subset \mU_y$
  (thus $\mU_y \in \{U_1^y, ..., U_{n_y}^y\}^*$). By the definition of
   $\mU_y$, for each $U_i^y$, there is a $V_i^y$ such that
   $y \in V_i^y$ and $U_i^y \times V_i^y \in \mW$. Define
   $V_y = \bigcap_{i = 1}^{n_y} V_i^y$. Since $Y$ is Lindel\"of,
   there is a sequence $(y_k)_{k \in \w}$ such that
   $Y \subset \bigcup_{k \in \w}V_{y_k}$. Let us prove that
    $(\{U_1^{y_k}, ..., U_{n_{y_k}}^{y_k}\}^*)_{k \in \w}$ is a covering
    of $\mY$. Let $\mU_y \in \mY$. Let $k \in \w$ such that
    $y \in V_{y_k} = \bigcap_{i = 1}^{n_{y_k}} V_i^{y_k}$. So, by definition
    of $\mU_y$, each $U_i^{y_k} \in \mU_y$, which  means that
    $\mU_y \in \{U_1^{y_k}, ..., U_{n_{y_k}}^{y_k}\}^*$  and thus this completes our proof.
\end{proof}

Now we will investigate when the pL property applied to $\mY$ is
enough to guarantee that $\mW$ has a countable subcovering. For this,
the following definitions will be helpful:

\begin{definition}
  Let $\mW$ be a covering of $X \times Y$ made by basic open sets.
   Then we say that $\mW$ is an {\bf injective covering} if whenever
   $A \times B, A' \times B' \in \mW$, then $A = A'$ implies $B = B'$.
    We say that $\mW$ is an {\bf $\w$-injective covering} if, for every $A$,
    the set $\{B: A \times B \in \mW\}$ is at most countable.
\end{definition}

\begin{lemma}\label{for injective is ok}
  Let $X$ and $Y$ be Lindel\"of spaces and let
   $\mW$ be an $\w$-injective covering of $X \times Y$.
   If $X$ has the pL property, then $\mW$ has a countable subcovering.
\end{lemma}
\begin{proof}
Let $\mathcal Y$ be such as before Proposition \ref{mY is Lindelof}.
 The latter implies that $\mY$ is a Lindel\"of collection.
  Let $(A_n)_{n \in \w}$ be given by the pL property. For each $A_n$,
  let $(B_n^m)_{m \in \w}$ be an enumeration of all open sets $B\subset Y$
   such that
   $A_n \times B\in \mW$. We will show that
    $X \times Y \subset \bigcup_{n,m \in \w} A_n \times B_n^m$.
    Let $(x, y) \in X \times Y$. Let $n \in \w$ be such that
     $x \in A_n \in \mU_y$. Note that, since $A_n \in \mU_y$, there is a
     $B$ such that $y \in B$ and $A_n \times B \in \mW$. Thus
      there is an $m$ such that $B_n^m = B$, and therefore
       $(x, y) \in A_n \times B_n^m$.
\end{proof}

Note that if $X \times Y$ is Lindel\"of, then every open covering of
$X \times Y$ has an $\w$-injective refinement (just go for a
countable refinement made by basic open sets). Therefore, the
following result is an easy consequence of the previous ones:

\begin{theorem}\label{external carac}
  Let $X$ be a Lindel\"of space. Then $X$ is productively
  Lindel\"of if, and only if, $X$ has the pL property and,
  for every Lindel\"of $Y$ and every $\mW$ covering for $X \times Y$,
  there is an $\w$-injective refinement for $\mW$.
\end{theorem}

Thus, for every class of spaces such that it is always possible to find $\w$-injective refinements as above, the productively Lindel\"of property is equivalent to the pL property.
 In the following we will discuss such classes of spaces.

\subsection{Getting $\w$-injective refinements}

The objective of this section is to show that it is possible to find $\w$-injective refinements for every regular space. But some of the results of
this section will be also used in the next section where we consider non-regular spaces.

\begin{lemma}\label{size of B}
  Let $X$ and $Y$ be any spaces, where $Y$ is Lindel\"of and $X$ is $T_1$. Let $\mW$ be an open covering of $X \times Y$ and let $\mB$ be a base for $X$. Then there is an open refinement $\mR$ of $\mW$ such that every element of $\mR$ is of the form $B \times C$ where $B \in \mB$ and, for every $B \in \mB$, $|\{C: B \times C \in \mR\}| \leq |B| + \aleph_0$.
\end{lemma}

\begin{proof}
For every isolated point $x\in X$ let us fix a countable open covering $\mathcal C_x$
of $Y$ such that for each element of $\{\{x\}\times C:C\in\mathcal C_x\}$
there is an element of $\mathcal W$ containing it. Set $\mathcal R_0=\{\{x\}\times C:x\in X$ is  isolated
and $C\in\mathcal C_x\}$.

  Now let $\kappa$ be an infinite cardinal.
We define $X_\kappa = \{x \in X: \kappa$ is the least $\rho$ such that there is a $B \in \mB$ such that $x \in B$ and $|B| = \rho\}$. Let $\mB_\kappa = \{B \in \mB: |B| = \kappa\}$.

  Let $\{x_\xi: \xi < \lambda\}$ be an enumeration of $X_\kappa$. Let $M_0$ be an elementary submodel
 such that $x_0, X, Y, \mW, \mB \in M_0$, and $|M_0| = \kappa$. Then   $A \subset M_0$ for every
$A \in M_0$ with $|A| \leq \kappa$. Let
$$\mR_0^\kappa = \{B \times C \in M_0: B \times C \text{ refines $\mW$ and }B \in \mB_\kappa\}.$$
Note that $|\mR_0^\kappa| \leq \kappa$.

\begin{claim}
  If $x \in X_\kappa$ and there is a $B \in \mB_\kappa$ such that $B \times C \in \mR_0^\kappa$ for some $C$, then $\{x\} \times Y \subset \bigcup \mR_0^\kappa$.
\end{claim}

\begin{proof}
  Let $B \in \mB_\kappa$ be such that $x \in B$ and $B \times C \in \mR_0^\kappa$.
Since $|B| = \kappa$, $B \subset M_0$. Thus, since $Y$ is Lindel\"of, there is a sequence
 $(B_n \times C_n)_{n \in \w}$ covering $\{x\} \times Y$, refining $\mW$,
and such that $B_n\in\mathcal B_\kappa$ for all $n\in\w$.
By elementarity, we may suppose that this sequence is in $M_0$ and, therefore, each $B_n \times C_n \in \mR_0^\kappa$ which proves the claim.
\end{proof}

Now let $x_\xi$ be the first one such that $x_\xi$ is not covered by any $B$ such that
 $B \times C \in \mR_0^\kappa$ for some $C$. Let $M_\xi$ be an elementary submodel as before, but this time containing $x_\xi$. Define
$$R_\xi^\kappa = \{B \times C \in M_\xi: B  \in \mB_\kappa, B  \smallsetminus A \neq \emptyset \text{ and } B \times C \text{ refines } \mW\}$$
where $A$ is the collection of all $x \in X$ that are covered by some $B$ that $B \times C \in \mR_0^\kappa$ for some $C$.

Note that, again, $|R_\xi^\kappa| \leq \kappa$. Also, if $B \times C \in \mR_0^\kappa$ and $B' \times C' \in \mR_\xi^\kappa$, then $B \neq B'$. Finally, note that the analogous of the Claim for $\mR_\xi^\kappa$ also holds. Thus, we can proceed like this until there is no $x_\eta \in X_\kappa$ not covered. Then, define $\mR_\kappa = \bigcup_{\xi < \lambda} \mR_\xi^\kappa$ (if $\mR_\xi^\kappa$ was not defined, just let it be empty). Finally, note that $\mR = \bigcup_{\kappa < |X|} \mR_\kappa$ is the refinement we were looking for.
\end{proof}

\begin{lemma} \label{l4_3}
  Let $X$ be a $T_1$ space without isolated points  and  $\mB$ be a base of $X$
 such that there are no different  $A, B \in \mB$
with $|A \Delta B| $ finite. Then for every Lindel\"of space $Y$ and every $\mW$
covering for $X \times Y$ there is an injective  refinement of $\mathcal W$.
\end{lemma}

\begin{proof}
  First, applying Lemma \ref{size of B}, we may suppose that every element of $\mW$ is of the
 form $B \times C$ with $B \in \mB$ and such that
$$|\{C: B \times C \in \mathcal W\}| \leq |B|$$
for every $B \in \mB$. Let $B \in \mB$. For each $C$ such that $B \times C \in \mW$,
select different $b_C, d_C \in B$ in such a way that $\{b_C,d_C\}\cap \{b_{C'},d_{C'}\} =\emptyset$
 if $C \neq C'$. Note that $\{(B \smallsetminus \{b_C\}) \times C, (B \smallsetminus \{d_C\}) \times C\}$ refines $B \times C$.
Also, note that if we repeat this process with all the infinite elements of $\mB$,
we get an injective refinement.
\end{proof}

\begin{lemma} \label{l4_4}
  Any regular space with no isolated points has a base $\mB$ such
 that there are no different  $A, B \in \mB$ such that $|A \Delta B| <\aleph_0$.
\end{lemma}

\begin{proof}
  Note that the regular open sets form a base. Also, note that, if $A$ and $B$ are different regular open sets,
then $A \Delta B \supset (A \smallsetminus \overline B) \bigcup (B \smallsetminus \overline A)$,
 and the latter is    a
non-empty open set.  Since $X$ has no isolated points, this set is infinite.
\end{proof}

\begin{lemma}  \label{l4_5}
  Let $X$ be a scattered $T_1$ space and let $Y$ be a Lindel\"of space. Let $\mW$ be an open covering for $X \times Y$. Then there is an $\w$-injective  open refinement of $\mW$.
\end{lemma}

\begin{proof}
  Let $X_0$ be the set of all isolated points of $X$. Then, let $X_\xi$ be the set of all isolated points of $X \smallsetminus \bigcup_{\eta < \xi} X_\eta$.

  For every $x \in X_\xi$, since $Y$ is Lindel\"of, we can find $(A_n^x \times B_n^x)_{n \in \w}$
 that is a refinement of $\mW$ that covers $\{x\} \times Y$. We may also assume that
 $A_n^x\smallsetminus \{x\} \subset \bigcup_{\eta < \xi} X_\eta$ for all $x \in X_\xi$ and $n\in\omega$, i.e.,  $X_\xi \cap A_n^x = \{x\}$.
Note that the union of all $(A_n^x \times B_n^x)$'s covers the whole $X \times Y$ and that this is $\w$-injective since $A_n^x \neq A_m^y$ if $x \neq y$.
\end{proof}

Combining Lemmata ~\ref{l4_3}, \ref{l4_4}, \ref{l4_5},  we obtain our main result:

\begin{theorem}\label{main theorem}
  For every regular Lindel\"of space, $X$ is productively Lindel\"of if, and only if, $X$ has the $pL$ property.
\end{theorem}

Finally, it is worth to mention that the
regularity hypothesis  in the characterization above was used only
to guarantee refinements of open coverings with as many different
open sets as necessary for the proof. In general, we don't know
whether the regularity can be dropped, the next subsection is
devoted to the particular situation when we do not need it.

\subsection{Non-regular spaces}

Even for non regular spaces we can sometimes obtain the same characterization. For a space $X$ we denote by $kc(X)$ the minimal cardinality of a covering of $X$ by its compact subspaces. Note that if
$Y$ is Lindel\"of then any open covering of $X\times Y$ has a subcovering
of size $\leq kc(X)$.

\begin{lemma}\label{injective refinement}
  Let $X$ and $Y$  be Lindel\"of
   spaces, with $X$ being $T_1$ and such that every open subset of $X$ has cardinality at
   least $kc(X)$. Then every open covering $\mW$ of $X \times Y$ has an
   injective
   refinement $\mathcal W'$ consisting of standard basic open subsets of
   $X\times Y$.
   Moreover, if $\mathcal B$ is a base for $X$, then we may
   additionally assume that any $W\in\mathcal W'$ has the form $(B\setminus\{x\})\times V$
   for some $B\in\mathcal B,$ $x\in X$, and open $V\subset Y$.
\end{lemma}
\begin{proof}
  We may assume that $\mW$ consists of standard
  basic open  subsets of $X\times Y$ and  $|\mW| \leq kc(X) = \kappa$.
   Let us write  $\mW$ in the form $  (A_\xi \times B_\xi)_{\xi < \kappa}$.
   We will construct two sequences
   $(x_\xi^1)_{\xi < \kappa}$, $(x_\xi^2)_{\xi < \kappa}$ of elements of $X$ in
    such a way that for every $\xi < \kappa$ we have the following properties:
  \begin{enumerate}
  \item $x_\xi^i \in A_\xi$ for $i = 1, 2$;
  \item $x_\xi^1 \neq x_\xi^2$;
  \item $A_\xi \smallsetminus \{x_\xi^i\} \neq A_\eta \smallsetminus \{x_\eta^j\}$ for every $i, j \in \{1, 2\}$ and every $\eta < \xi$.
  \end{enumerate}
This is easily done by induction
 by our assumption on the cardinality of $A_\xi$. Now note that $\{(A_\xi
\smallsetminus \{x_\xi^i\}) \times B_\xi: \xi < \kappa$ and $i = 1,
2\}$ is the refinement  we were looking for.
\end{proof}

\begin{corollary}\label{carac kcX}
  Let $X$ be a Lindel\"of $T_1$ space such that every open subset
   has cardinality at least $kc(X)$. Then $X$ is productively Lindel\"of if,
   and only if, $X$ has the pL property.
\end{corollary}

\begin{proof}
  It follows directly from Theorem \ref{external carac} and Lemma \ref{injective refinement}.
\end{proof}

\begin{proposition}
  Let $X$ be a Lindel\"of $T_1$ space of size $\aleph_1$ and let $Y$ be a Lindel\"of space. Then, for every $\mW$ covering for $X \times Y$, there is an $\w$-injective refinement.
\end{proposition}

\begin{proof}
  Let $X'$ be the collection of all points of $X$ that have a countable neighborhood. Let $\mW'$ be a an $\w$-injective refinement for $X' \times Y$ given by the Lemma \ref{size of B}. Note that $X'' = X \smallsetminus X'$ is closed and, therefore, Lindel\"of. Thus we can do the same argument as in the proof of Lemma \ref{injective refinement} and get $\mW''$ an injective refinement for $X'' \times Y$ (note that for each open set, in the argument of the Lemma we can remove points from $X'$). Thus, $\mW' \cup \mW''$ is the refinement we were looking for.
\end{proof}

\begin{corollary}
  Let $X$ be a Lindel\"of $T_1$ space of size $\aleph_1$. Then $X$ is productively Lindel\"of if, and only if, $X$ has the pL property.
\end{corollary}


\section{About Alster coverings}

For this section and the next one, we need some more definitions:

\begin{definition}
  Let $X$ be a topological space. We say that $\mG$ is an {\bf Alster covering} for $X$ if each $G \in \mG$ is a $G_\delta$ set and, for every compact $K \subset X$, there is a $G \in \mG$ such that $K \subset G$. We say that $X$ is an {\bf Alster space} if every Alster covering has a countable subcovering.
\end{definition}

In this section, we present a generalization of the following:

\begin{theorem}[Alster \cite{Alster1986}]\label{alst_orig}
  Suppose CH. Let $X$ be a Tychonoff space with weight $\leq \aleph_1$. Then $X$ is productively Lindel\"of if, and only if, for every Alster covering, there is a countable subcovering, \emph{i.e.}, $X$ is an Alster space.
\end{theorem}

In the following, we use the standard notation $I=[0,1]$. Our main result in this section is the following:

\begin{theorem} \label{alst_general}
  Let $X\subset I^{\w_1}$ be a productively Lindel\"of Menger space and
 $\mG$ be an Alster covering of $X$ of size $|\mG|=\w_1$ consisting of
Lindel\"of $G_\delta$-subsets of $I^{\w_1}$. Then there exists $\mG'\in [\mG]^{\w}$
such that $X\subset\bigcup\mG'$.
\end{theorem}

First, note that if a space is an Alster space, then it is productively Lindel\"of. This is the ``easy'' part of Theorem \ref{alst_orig} in \cite{Alster1986}. Also, note that any Tychonoff space $X$ with weight $\leq \aleph_1$ can be viewed as a subspace of $I^{\w_1}$ and, under CH, given any Alster covering, there is a refinement of size $\aleph_1$ for this covering made by compact $G_\delta$-subsets of $I^{\w_1}$. Finally, under CH, every productively Lindel\"of space is Menger (\cite[Proposition~3.1]{RepZdo12}), since there is a Michael space. With this, we obtain that Theorem \ref{alst_orig} follows from Theorem \ref{alst_general}.

Before proving Theorem~\ref{alst_general} let us draw some corollaries
from it.

\begin{corollary} \label{cor1}
Suppose that there exists a Michael space.
Let $X\subset I^{\w_1}$ be a productively Lindel\"of  space and
 $\mG$ be an Alster covering of $X$ of size $|\mG|=\w_1$ consisting of
Lindel\"of $G_\delta$-subsets of $I^{\w_1}$. Then there exists $\mG'\in [\mG]^{\w}$
such that $X\subset\bigcup\mG'$.
\end{corollary}

\begin{proof}
By \cite[Proposition~3.1]{RepZdo12} there exists a Michael space if and only if
all productively Lindel\"of spaces are Menger.
\end{proof}

\begin{corollary} \label{cor2}
Suppose that $\mathfrak b=\w_1$ or $\mathit{cov}(\mM)=\mathfrak d$.
Let $X\subset I^{\w_1}$ be a productively Lindel\"of  space and
 $\mG$ be an Alster covering of $X$ of size $|\mG|=\w_1$ consisting of
Lindel\"of $G_\delta$-subsets of $I^{\w_1}$. Then there exists $\mG'\in [\mG]^{\w}$
such that $X\subset\bigcup\mG'$.
\end{corollary}
\begin{proof}
 If  $\mathfrak b=\w_1$ or $\mathit{cov}(\mM)=\mathfrak d$, then there exists a
Michael space. More precisely, its existence under
 $\mathfrak b=\w_1$ follows by an almost literal repetition of Michael's
  proof \cite{Mic63}  that $\w^\w$ is not productively Lindel\"of
under CH, and the case $\mathit{cov}(\mM)=\mathfrak d$ has been treated in
\cite[Theorem~2.2]{Moo99}.
\end{proof}

We shall divide the proof of Theorem~\ref{alst_general} into a sequence of
lemmas. For every $\alpha\in\w_1$ we shall denote by $p_\alpha$ the projection
map from $I^{\w_1}$ to $I^\alpha$, i.e., $p_\alpha:(x_\xi)_{\xi<\w_1}\mapsto (x_\xi)_{\xi<\alpha}$.
Let us denote by $\mB_\alpha$ the family of open subsets of $I^{\w_1}$
of the form $\prod_{\xi\in\w_1}U_\xi$, where $U_\xi=I$ for all $\xi\not\in F$
for some finite $F\subset\alpha$, and $U_\xi$ is an interval with rational
end-points for all $\xi\in F$. Thus $\mB:=\bigcup_{\alpha\in\w_1}\mB_\alpha$ is the standard base
for the topology on $I^{\w_1}$.

The following fact may be thought of as folklore.

\begin{lemma} \label{standard_g_deltas}
A $G_\delta$ subset $G$ of $I^{\w_1}$ is  Lindel\"of if, and only if,
there exists $\alpha<\w_1$
such that $G=p_\beta^{-1}[p_\beta[G]]$ for all $\beta\geq\alpha$.
\end{lemma}
\begin{proof}
For the ``if'' part note that $G=p_\alpha^{-1}[p_\alpha[G]]$ implies
$G=p_\alpha[G]\times I^{\w_1\setminus\alpha}$, i.e., $G$ is a product of a metrizable
separable space and a compact space. Such products are obviously Lindel\"of.

Let us prove now the ``only if'' part.
Write $G$ in the form $\bigcap_{n\in\w}U_n$, where $U_n$ is open.
Given any  $y\in G$ and $n\in\w$, fix $B(y,n)\in\mB$ such that $y\in B(y,n)\subset U_n$.
Let $\alpha(y,n)$ be such that $B(y,n)\in\mB_{\alpha(y,n)}$.
Since $G$ is Lindel\"of there exists a countable $Y\subset G$ such that
$G\subset\bigcup_{y\in Y}B(y,n)$ for all $n\in\w$.
Then $G=\bigcap_{n\in\w}\bigcup_{y\in Y}B(y,n)$. It is easy to see that
$\alpha=\sup\{\alpha(y,n):y\in Y,n\in\w\}+1$ is as required.
\end{proof}

We shall need the following result which is a direct consequence
of \cite[Corollary~2.5]{AlaAurJunTal11}.

\begin{proposition} \label{alaaurjuntal}
Let $Z$ be a metrizable space and $\{A_\xi:\xi<\w_1\}$ be
an increasing covering of $Z$ such that for every compact $K\subset Z$
there exists $\alpha\in\w_1$ with the property $K\subset A_\alpha$.
If $A_\xi\neq Z$ for all $\xi$, then $Z$ is not productively Lindel\"of.
\end{proposition}

We shall call $A\subset I^{\w_1}$ \emph{big}
if $A\cap X$ is not covered by any countable subfamily of $\mG$.
If $X$ is not big then there is nothing to prove. So we shall assume that
$X$ is big in the sequel.

\begin{lemma} \label{big_compact}
For every big closed subspace $Z$ of $ X$ and $\alpha\in\w_1$
there exists a compact $K\subset p_\alpha[Z]$ such that
$Z\cap p_\alpha^{-1}[K]$ is big.
\end{lemma}
\begin{proof}
Let $\mG=\{G_\xi:\xi<\w_1\}$ and $A_\beta=
p_\alpha[Z]\setminus p_\alpha[Z\setminus \bigcup_{\xi<\beta}G_\xi]$.
If there is no compact $K\subset p_\alpha[Z]$ such that
$Z\cap p_\alpha^{-1}[K]$ is big, then
$\{A_\xi:\xi<\w_1\}$ is easily seen to be
an increasing covering of $p_\alpha[Z]$ such that for every compact $K\subset p_\alpha[Z]$
there exists $\xi\in\w_1$ with the property $K\subset A_\xi$,
and $A_\xi\neq p_\alpha[Z]$ for all $\xi$ because $Z$ is big.
Now Proposition~\ref{alaaurjuntal} implies that $p_\alpha[Z]$ is not productively
Lindel\"of, a contradiction.
\end{proof}

For a topological space $Z$ we shall denote by $\O(Z)$ the family of all
open coverings of $Z$. For a subspace $T$ of $Z$ we shall denote by
$\cl_Z(T)$ the closure of $T$ in $Z$.

\begin{lemma} \label{small_covers}
Let $A\subset X$ be a big Lindel\"of subspace.
Then there exists $\alpha\in\w_1$ such that for every $\beta>\alpha$
there exists an open covering $\mW_\beta$ of $p_\beta[A]$ such that
$A\setminus p_\beta^{-1}[\cup\mV]$ is big for all $\mV\in [\mW_\beta]^{<\w}$.
\end{lemma}
\begin{proof}
Suppose that contrary to our claim
the set $\Lambda$ of all $\alpha\in\w_1$ such that
for every  open covering $\mW $ of $p_\alpha[A]$ there exists
$\mV_\mW \in [\mW]^{<\w}$ for which  $A\setminus p_\alpha^{-1}[\cup\mV_\mW]$ is  \emph{not} big,
is cofinal in $\w_1$.  For every $\alpha\in\Lambda$ the set
$K_\alpha=\bigcap_{\mW\in\O(p_\alpha[A])}\cl_{p_\alpha[A]}(\cup\mV_\mW)$ is a compact subspace of
$p_\alpha[A]$. Since $p_\alpha[A]$ is metrizable separable,
there exists a countable $\mathsf W_\alpha\subset\O(p_\alpha[A])$
such that $K_\alpha=\bigcap_{\mW\in\mathsf W_\alpha}\cl_{p_\alpha[A]}(\cup\mV_\mW)$.
Since for every $\mW\in\mathsf W_\alpha$ the set
$A\setminus p_\alpha^{-1}[\cup\mV_\mW]$ is not big, so is the set
$$ A\setminus p_\alpha^{-1}[K_\alpha] =
 A\setminus \bigcap_{\mW\in\mathsf W_\alpha}p_\alpha^{-1}[\cup\mV_\mW] =
\bigcup_{\mW\in\mathsf W_\alpha} (A\setminus p_\alpha^{-1}[\cup\mV_\mW]) $$
because non-big sets are obviously closed under countable unions.
Therefore
$$ \bigcup _{\alpha\in\Lambda\cap\gamma} (A\setminus p_\alpha^{-1}[K_\alpha]) =
 A\setminus \bigcap _{\alpha\in\Lambda\cap\gamma} p_\alpha^{-1}[K_\alpha]
$$
is not big for every $\gamma\in\w_1$, and hence
$A\cap \bigcap _{\alpha\in\Lambda\cap\gamma} p_\alpha^{-1}[K_\alpha]$ is big
for every $\gamma\in\w_1$. Note that $K:=\bigcap _{\alpha\in\Lambda} p_\alpha^{-1}[K_\alpha]$
is a compact subspace of $I^{\w_1}$. Since $K_\alpha\subset p_\alpha[A]$
for all $\alpha\in\Lambda$ and $A$ is Lindel\"of, we have that $K\subset A$
(suppose there is a $k \in K \setminus A$. Note that $(U_\alpha)_{\alpha < \w_1}$ is an open covering for $A$, where each $U_\alpha = \{a \in A: a(\alpha) \neq k(\alpha)\}$. But such a covering cannot have a countable subcovering since $p_\alpha(k)\in p_\alpha[A]$ for every $\alpha\in\w_1$).
Thus there exists $G\in\mG$ such that $K\subset G$, and hence
$\{A\setminus  p_\alpha^{-1}[K_\alpha]:\alpha\in\Lambda\}$ is an open covering of
$A\setminus G$. Since $A$ is Lindel\"of, so is $A\setminus G$ being
an $F_\sigma$-subset of a Lindel\"of space, and hence
$A\setminus G\subset \bigcup _{\alpha\in\Lambda\cap\gamma} (A\setminus p_\alpha^{-1}[K_\alpha])$
for some $\gamma<\w_1$. Therefore
$G\supset A\cap \bigcap _{\alpha\in\Lambda\cap\gamma} p_\alpha^{-1}[K_\alpha]$
which means that the latter set cannot be big, a contradiction.
\end{proof}

Let us fix a map $\psi:\mG\to\w_1$ such that
$\psi^{-1}(\alpha)$ is countable for all
$\alpha$, and $\psi(G)$ satisfies Lemma~\ref{standard_g_deltas}
for $G$, i.e., $G=p_\beta^{-1}[p_\beta[G]]$ for all $\beta\geq \psi(G)$.
For every $\alpha$ let us consider the set
$C_\alpha=\{z\in I^{\w_1}:p_\beta(z)\in p_\beta[X]$ for all $\beta<\alpha$,
$p_\alpha(z)\not\in p_\alpha[X]$, and $z\not\in G$ for all $G\in\mG$ such that $\psi(G)<\alpha\}$.
Note that $C_\alpha$ depends on $\psi$.

We shall need the following game of length $\omega$:
 In the $n$th  round Player $I$ chooses an open covering
$\mU_n$ of $X$, and Player $\mathit{II}$ responds by choosing a
finite $\mV_n\subset \mU_n$. Player $\mathit{II}$ wins the game if $\bigcup_{n\in\omega}
\bigcup\mV_n =X$. Otherwise, Player $I$ wins.  We shall call this game
 \emph{the Menger game} on $X$.
 Since $X$ is Menger, Player $I$ has no winning strategy in this game on
$X$, see \cite{Hur25} or \cite[Theorem~13]{COC1}.

\begin{lemma} \label{c_alpha}
The set $\Lambda=\{\alpha:C_\alpha\neq\emptyset\}$ is unbounded in $\w_1$.
\end{lemma}
\begin{proof}
Given $\alpha_0\in\w_1$, we shall find $\alpha>\alpha_0$ such that
$C_\alpha\neq\emptyset$. Now we shall describe
a strategy  of Player $I$ in the Menger game on $X$.

{\it Round 0.} \ Set $A_0=X$. Since $A_0$ is big and productively Lindel\"of,
by Lemma~\ref{big_compact} there exists a compact $K_0\subset p_{\alpha_0}[A_0]$
such that $A_0':=A_0\cap p_{\alpha_0}^{-1}[K_0]$ is big. Let  $G_0\in\mG$ be such
that $\psi(G_0)<\alpha_0$. Since $A_0'\setminus G_0$ is big and non-big sets are
closed under countable unions,  there exists an open set $S_0\supset G_0$
such that  $p_\beta^{-1}[p_\beta[S_0]]=S_0$ for all $\beta\geq \psi(G_0)$,
and $A_0'':=A_0'\setminus S_0$ is big. By Lemma~\ref{small_covers}
applied to $A_0''$ there exists $\alpha_1>\alpha_0$ and $\mW_0\in\O(p_{\alpha_1}[A_0''])$
such that $A_0''\setminus p_{\alpha_1}^{-1}[\cup\mV]$ is big for all
$\mV\in[\mW_0]^{<\w}$. Then Player $I$ starts by choosing the
open covering $\mU_0=\{X\setminus A_0''\}\bigcup\{p_{\alpha_1}^{-1}[W]:W\in\mW_0\}$
of $X$. Suppose that Player $II$ replies by choosing
$\{X\setminus A_0''\}\bigcup\{p_{\alpha_1}^{-1}[W]:W\in\mV_0\}$ for some
finite $\mV_0\subset\mW_0$. Then we set $A_1=A_0''\setminus p_{\alpha_1}^{-1}[\cup\mV_0]$.
It follows from the above that $A_1$ is a big closed subspace of $X$.

{\it Round $n$}. Suppose that after $n-1$ rounds  the set $A_n\subset X$
of those $x\in X$ which have not yet
been covered by the choices of Player $II$, is closed\footnote{It is always closed by the definition
of the game.} and big. Suppose also that in the course of the previous rounds
player $I$ has constructed ordinals $\alpha_0<\alpha_1<\cdots<\alpha_{n}$.
Then player $I$ acts basically in the same way as in round $0$.
For the sake of completeness we repeat the argument.
Since $A_n$ is big and productively Lindel\"of,
by Lemma~\ref{big_compact} there exists  a compact $K_n\subset p_{\alpha_n}[A_n]$
such that $A_n':=A_n\cap p_{\alpha_n}^{-1}[K_n]$ is big. Let  $G_n\in\mG$ be such
that $\psi(G_n)<\alpha_n$. Since $A_n'\setminus G_n$ is big and non-big sets are
closed under countable unions,  there exists an open set $S_n\supset G_n$
such that  $p_\beta^{-1}[p_\beta[S_n]]=S_n$ for all $\beta\geq \psi(G_n)$,
and $A_n'':=A_n'\setminus S_n$ is big. By Lemma~\ref{small_covers}
applied to $A_n''$ there exists $\alpha_{n+1}>\alpha_n$ and $\mW_n\in\O(p_{\alpha_{n+1}}[A_n''])$
such that $A_n''\setminus p_{\alpha_{n+1}}^{-1}[\cup\mV]$ is big for all
$\mV\in[\mW_n]^{<\w}$. Then Player $I$ plays by choosing the
open covering $\mU_n=\{X\setminus A_n''\}\bigcup\{p_{\alpha_{n+1}}^{-1}[W]:W\in\mW_n\}$
of $X$. Suppose that Player $II$ replies by choosing
$\{X\setminus A_n''\}\bigcup\{p_{\alpha_{n+1}}^{-1}[W]:W\in\mV_n\}$ for some
finite $\mV_n\subset\mW_n$. Then we set $A_{n+1}=A_n''\setminus p_{\alpha_{n+1}}^{-1}[\cup\mV_n]$.
It follows from the above that $A_{n+1}$ is a big closed subspace of $X$.

In addition, by choosing $G_n$ in the $n$th round, Player $I$ makes sure
that each $G\in\mG$ with $\psi(G)<\sup_{k\in\w}\alpha_k$ has been chosen, using
 some straightforward bookkeeping. This completes our definition of a strategy of
player $I$.

Since $X$ is Menger, this strategy cannot be winning,
and hence there is a play in which Player $I$ uses the strategy described above and
loses. Let
$$\langle \alpha_n,A_n,A_n',A_n'',\mW_n,\mV_n,K_n, G_n, S_n:n\in\w\rangle $$
be the corresponding objects constructed in the course of this play.
Since this play has been lost by Player $I$ we have
$$X=\bigcup_{n\in\w}(X\setminus A_n'')\cup\bigcup_{n\in\w}p_{\alpha_{n+1}}^{-1}[\cup\mV_n].$$
Letting $\alpha=\sup_{n\in\w}\alpha_n$,
we claim that $C_\alpha\supset K:=\bigcap_{n\in\w}p_{\alpha_n}^{-1}[K_n]$.
Note that this would prove our lemma as $K$ is non-empty
because $\langle p_{\alpha_n}^{-1}[K_n]:n\in\w\rangle$ is a decreasing sequence of compact subspaces
of $I^{\w_1}$. Let us fix $z\in K$.

Given any $\beta<\alpha$, find $n$ with $\beta<\alpha_n$. Then
  $p_{\alpha_n}(z)\in  K_n\subset p_{\alpha_n}[A_n]\subset p_{\alpha_n}[X]$,
and hence $p_\beta(z)\in p_{\beta}[X]$.

Let us fix $G\in \mG$ with $\psi(G)<\alpha$. By the  requirement on the strategy of
the Player $I$ we have made at the end of its definition, we have that
$G=G_n$ for some $n\in\w$. Then $p_{\alpha_{n+1}}(z)\in K_{n+1}\subset p_{\alpha_{n+1}}[A_{n+1}]$
and $A_{n+1}\cap S_n=A_{n+1}\cap p_{\alpha_{n+1}}^{-1}[p_{\alpha_{n+1}}[S_n]]=\emptyset$
by  the construction. Therefore $p_{\alpha_{n+1}}(z)\not\in p_{\alpha_{n+1}}[S_n]$
and hence $z\not\in S_n\supset G_n$.

It suffices to show that $p_\alpha(z)\not\in p_\alpha[X]$. Suppose to the contrary
that $p_\alpha(z)=p_\alpha(x)$ for some $x\in X$.
Two cases are possible.

$1.$ \ $x\in X\setminus A_n''$ for some $n\in\w$. Then
$x\in X\setminus A_{n+1}$ since $A_{n+1}\subset A_n''$.
By the construction of $A_k$  we get that
$p_{\beta}[A_{k}\setminus A_{k+1}]=
p_{\beta}[A_k]\setminus p_{\beta}[A_{k+1}]$ for all $k\in\w$
and $\beta\geq \alpha_{k+1}$.
Indeed to get $A_{k+1}$ we make several steps, and at each of them
we remove from $A_k$ a set $T$ such that $T=p_\beta^{-1}[p_\beta[T]]$
for all $\beta\geq \alpha_{k+1}$. Since $X=A_0$, we have
that $ X\setminus A_{n+1}=\bigcup_{k\leq n}(A_k\setminus A_{k+1}), $
and hence
\begin{eqnarray*}
p_{\alpha_{n+1}}(x)\in p_{\alpha_{n+1}}[X\setminus A_{n+1} ] =p_{\alpha_{n+1}}[\bigcup_{k\leq n}(A_k\setminus A_{k+1})]= \\
=
\bigcup_{k\leq n} p_{\alpha_{n+1}}[A_k]\setminus p_{\alpha_{n+1}}[A_{k+1}]=
p_{\alpha_{n+1}}[X]\setminus p_{\alpha_{n+1}}[A_{n+1}].
\end{eqnarray*}
However, $p_{\alpha_{n+1}}(x)=p_{\alpha_{n+1}}(z)\in K_{n+1}\subset p_{\alpha_{n+1}}[A_{n+1}],$
a contradiction.

$2.$ \ $x\in p_{\alpha_{n+1}}^{-1}[\cup\mV_n]$ for some $n\in\w$. Then
$x\in X\setminus A_{n+1}$ because $A_{n+1}=A_n''\setminus p_{\alpha_{n+1}}^{-1}[\cup\mV_n] $
by the construction, and we have already seen in case $1$
that $x\in X\setminus A_{n+1}$ leads to a contradiction.
This completes our proof.
\end{proof}

Let us denote by $Y$  the union $\bigcup_{\alpha\in\w_1}C_\alpha$.

\begin{lemma} \label{y_lindelof}
$Y$ is Lindel\"of.
\end{lemma}
\begin{proof}
 Let $\mU$ be an open covering $Y$. Without loss of generality we may assume
that $\mU$ is closed under finite unions and taking open subsets of its elements.
Let $\mU_\alpha=\mU\cap\mB_\alpha$. It suffices to show that $Y\subset \cup\mU_\alpha$ for some
$\alpha$. Suppose that this is not the case, set $\mG'=\mG\cup \{\cup\mU_\alpha:\alpha\in\w_1\}$,
and extend $\psi$ to $\mG'$ by letting $\psi(\cup\mU_\alpha)=\alpha$.
We claim that $X$ is not covered by any countable subfamily of $\mG'$. Indeed,
fix such a subfamily $\mG''$ and find $\alpha$ with $\psi[\mG'']\subset\alpha$, i.e.,
$\mG''\subset\{\cup\mU_\beta:\beta<\alpha\}\cup \{G\in\mG:\psi(G)<\alpha\}$.
Pick $z\in Y\setminus\bigcup_{\beta<\alpha}\cup\mU_\beta=Y\setminus\cup\mU_\alpha$.
Then $z\in C_\gamma$ for some $\gamma>\alpha$ by the following
\begin{claim} \label{cl01}
$C_\xi\subset\bigcup_{\eta<\xi}\cup\mU_\eta=\mU_\xi$ for all $\xi<\w_1$.
\end{claim}
\begin{proof}
Given $u\in C_\xi$, set $K_u= \{p_\xi(u)\}\times I^{\w_1\setminus\xi}$
and note that $K_u$ is a compact subspace of $Y$. Since $K_u\subset\cup\mU$,
there exists a finite
$\mV\subset\mU$ and an open set
$U=\prod_{\eta\in\w_1}U_\eta$, where $U_\eta=I$ for all $\eta\not\in F$
for some finite $F\subset \w_1$, and $U_\eta$ is an interval with rational
end-points for all $\eta\in F$, with the following properties:
$$ u\in K_u\subset U\subset\cup\mV. $$
$U\in\mU$ by our convention. Moreover, $U_\eta$ must obviously be equal to $I$
for all $\eta\geq \xi$, and hence $U\in\mU_\xi$.
\end{proof}
We proceed with the proof of Lemma~\ref{y_lindelof}.
Since $z\in C_\gamma$ for some $\gamma>\alpha$, we have
$z\not\in G$ for all $G\in\mG$ with $\psi(G)<\alpha$, and hence
$z\not\in G$ for any $G\in\mG''$. Pick $x\in X$ such that
$p_\alpha(x)=p_\alpha(z)$ and note that
$x\not\in G$ for any $G\in\mG''$ because $G=p_\alpha^{-1}[p_\alpha[G]]$
for all such $G$.

Now repeating the proof of Lemma~\ref{c_alpha} for $\mG'$
and $\psi$ extended to $\mG'$, we get that
there exists $\alpha\in\w_1$ such that the set
$C_\alpha'=\{z\in I^{\w_1}:p_\beta(z)\in p_\beta[X]$ for all $\beta<\alpha$,
$p_\alpha(z)\not\in p_\alpha[X]$, and $z\not\in G$ for all $G\in\mG'$ such that $\psi(G)<\alpha\}$
is non-empty. However,
$C_\alpha'=C_\alpha\setminus\bigcup_{\beta<\alpha}\cup\mU_\beta=
 C_\alpha\setminus\cup\mU_\alpha$ by the definition of $\mG'$ and $\psi\uhr\mG'\setminus\mG$,
 and $C_\alpha\setminus\cup\mU_\alpha=\emptyset$ by
Claim~\ref{cl01}, a contradiction.
\end{proof}

\begin{lemma} \label{nbh_of_y}
 Let $U$ be an open neighborhood of $Y$. Then $X\setminus U\subset\cup\mH$ for some
countable $\mH\subset\mG$.
\end{lemma}
\begin{proof}
Let $\mU$ be the family of all open subsets of $U$ and $\mU_\alpha$ be the intersection $\mU\cap\mB_\alpha$.
In the proof of Lemma~\ref{y_lindelof} we have established that there exists $\alpha\in\w_1$
and $\mH\in[\mG]^\w$ such that $X\subset \cup\mU_\alpha\cup\bigcup\mH$. (More precisely,
by the last paragraph of the proof of Lemma~\ref{y_lindelof} we have  a
contradiction if $X\not\subset \cup\mU_\alpha\cup\bigcup\mH$ for all $\alpha$ and $\mH\in[\mG]^\w$).
Therefore $X\setminus U\subset X\setminus \cup\mU_\alpha\subset\cup\mH$.
\end{proof}

We can finish now the proof of Theorem~\ref{alst_general}
by deriving a contradiction from $X$ being big as follows.
Let $\mG=\{G_\alpha:\alpha<\w_1\}$ and pick $x_\alpha\in X\setminus\bigcup_{\beta<\alpha}G_\beta$.
Consider the space $Z=\{x_\alpha:\alpha<\w_1\}\cup Y$ with the following
topology $\tau$: all $x_\alpha$'s are isolated, and the basic open neighborhoods of
 $y\in Y$ have the form $U\cap Z$, where $U\ni y$ is open in $I^{\w_1}$.
 $(Z,\tau)$ is Lindel\"of: $(Y,\tau\uhr Y)$ is Lindel\"of by Lemma~\ref{y_lindelof},
and by Lemma~\ref{nbh_of_y} and the definition of $\tau$ we have that
 $Z\setminus U$ is countable for any open $U\in \tau$
containing    $Y$.
Now $X\times Z$ is not Lindel\"of when $Z$ is considered with the topology $\tau$
as $\{(x_\alpha, x_\alpha):\alpha<\w_1\}$ is a closed discrete uncountable subspace
of this product, a contradiction to $X$ being productively Lindel\"of.

\section{Weakly Alster spaces}

The inspiration for the results in this section is the characterization for Hurewicz spaces presented in Theorem 6 of \cite{Tall2011}:

\begin{theorem}[Tall]\label{TallHure}
  A regular Lindel\"of space $X$ is Hurewicz if, and only if, for every \v Cech-complete space $G \supset X$ there is a $\sigma$-compact space $K$ such that $X \subset K \subset G$.
\end{theorem}

The standard definition for Hurewicz spaces is in terms of a selection principle involving $\gamma$-coverings. It was proved in \cite[Theorem 7]{Tall2011} that every regular Alster space is Hurewicz. And, as in the previous section, the Alster property implies productively Lindel\"of. In this section, we present a new definition, called weakly Alster, that is in-between the Alster and productively Lindel\"of properties. It is not know by us if this new property is actually equivalent to any of the other two in general. But one advantage is that the relation between weakly Alster and Hurewicz properties is immediate because of Theorem \ref{TallHure}.

\begin{definition}
  Let $X$ be a regular space. We say that it is \emph{weakly Alster} if for every covering $\mG$ made by $G_\delta$ sets of $\beta X$ such that $\{G \cap X: G \in \mG\}$ is an Alster covering for $X$, there is a $\sigma$-compact space $Y \subset \beta X$ such that $X \subset Y \subset \bigcup \mG$.
\end{definition}

\begin{proposition}
  Every Alster space is weakly Alster.
\end{proposition}

\begin{proof}
  Let $X$ be an Alster space. Let $\mG$ be a covering for $X$ as in the definition of weakly Alster. We may suppose, taking a refinement if necessary,
that for each $G \in \mG$ there is a sequence $(A^G_n)_{n \in \w}$ of open sets in $\beta X$ such that $G = \bigcap_{n \in \w} A^G_n $ and
 $\mathit{cl}_{\beta X}(A^G_{n + 1}) \subset A^G_n$ for all $n$. In other words, we may assume that each $G\in\mG$ is compact.
 Since $X$ is Alster and $\{X\cap G:G\in\mG\}$ is an Alster cover of $X$,
there is a sequence $(G_n)_{n \in \w}$ of elements of $\mG$ such that $X \subset \bigcup_{n \in \w} G_n$.
This completes our proof.
\end{proof}

\begin{proposition}
  Every weakly Alster space is productively Lindel\"of.
\end{proposition}

\begin{proof}
  Let $X$ be a weakly Alster space and let $Z$ be a Lindel\"of space. Let $\mC$ be a basic open covering for $X \times Z$.
 We may suppose that, for each compact $K \subset X$ and each $z \in Z$, there are open sets $A_{K, z}, B_{K, z}$ such that
$K \times \{z\} \subset A_{K, z} \times B_{K, z} \in \mC$. Also, we may suppose that each $A_{K, z}$ is an open set in $\beta X$.
Since $Z$ is Lindel\"of, for each compact $K \subset X$, there is a sequence $(z_n)_{n \in \w}$
(depending on $K$) such that $K \times Z \subset \bigcup_{n \in \w} A_{K, z_n} \times B_{K, z_n}$.
Let $G_K = \bigcap_{n \in \w} A_{K, y_n}$. Let $Y \subset \beta X$ be a $\sigma$-compact such that $X \subset Y \subset \bigcup_{K \in \mK(X)} G_K$.

Fix $(y, z) \in Y \times Z$. Let $K$ be such that $y \in G_K$. Let $z_n$ be such that $z \in B_{K, z_n}$. Note that $(y, z_n) \in A_{K, z_n} \times B_{K, z_n}$. Since $Y$ is $\sigma$-compact, there is a countable subcovering for $Y \times Z$. Note that this is also a countable subcovering for $X \times Z$.
\end{proof}

Since under CH, productively Lindel\"of and Alster is the same for Tychonoff spaces with weight $\leq \aleph_1$, the following questions are natural:

\begin{question}
  Is every  productively Lindel\"of Tychonoff space a weakly Alster space? Under $CH$?
\end{question}

\begin{question}
  Is every weakly Alster space an Alster space under CH?
\end{question}

\begin{question}
  Is every Tychonoff weakly Alster space with weight $\leq \mathfrak c$ an Alster space?
\end{question}

\end{document}